\documentclass[10pt]{amsart}
\usepackage{latexsym, amsmath,amssymb, hyperref}

\usepackage{amsthm}

\numberwithin{equation}{section}

\newtheorem{theorem}{Theorem}

\newtheorem{lemma}[theorem]{Lemma}
\newtheorem{corr}[theorem]{Corollary}

\newtheorem{proposition}[theorem]{Proposition}
\newtheorem{deff}[theorem]{Definition}

\newcommand{\bth}{\begin{theorem}}
\newcommand{\ble}{\begin{lemma}}
\newcommand{\bcor}{\begin{corr}}

\newcommand{\bdeff}{\begin{deff}}

\newcommand{\bprop}{\begin{proposition}}
\newcommand{\ele}{\end{lemma}}
\newcommand{\ecor}{\end{corr}}
\newcommand{\edeff}{\end{deff}}

\newcommand{\eprop}{\end{proposition}}

\newcommand{\la}{\lambda}

\newcommand{\eps}{\varepsilon}

\newcommand{\supp}{\text{supp }}
\renewcommand{\Pi}{\varPi}

\newcommand{\dist}{{{\rm dist}}}

\newcommand{\R}{{\mathbb R}}

\thanks{ }

\usepackage{graphicx}

\begin{document}

\title[Inner product of eigenfunction]
{Inner product of eigenfunctions over curves and generalized periods for compact Riemannian surfaces}
%
%
%
%
%
\begin{abstract}  
We show that for a smooth closed curve $\gamma$ on a compact Riemannian surface without boundary, the inner product of two eigenfunctions $e_\la$ and $e_\mu$ restricted to $\gamma$, $|\int e_\la\overline{e_\mu}\,ds|$, is bounded by $\min\{\la^\frac12,\mu^\frac12\}$. Furthermore, given $0<c<1$, if  $0<\mu<c\la$, we prove that $\int e_\la\overline{e_\mu}\,ds=O(\mu^\frac14)$, which is sharp on the sphere $S^2$. These bounds unify the period integral estimates and the $L^2$-restriction estimates in an explicit way. Using a similar argument, we also show that the $\nu$-th order Fourier coefficient of $e_\la$ over $\gamma$ is uniformly bounded if $0<\nu<c\la$, which generalizes a result of Reznikov for compact hyperbolic surfaces, and is sharp on both $S^2$ and the flat torus $\mathbb T^2$. Moreover, we show that the analogs of our results also hold in higher dimensions for the inner product of eigenfunctions over hypersurfaces.
\end{abstract}

\author{Yakun Xi}

\address{Department of Mathematics, University of Rochester, Rochester, NY 14627}
\email{yxi4@math.rochester.edu}

\maketitle

\section{Introduction}
Let $e_\la$ denote the $L^2$-normalized eigenfunction on a compact, boundary-less Riemannian surface $(M,g)$, i.e., 
$$-\Delta_g e_\la=\la^2 e_\la, \quad \text{and } \, \, \int_M |e_\la|^2 \, dV_g=1.$$
Here $\Delta_g$ denotes the Laplace-Beltrami operator on $(M,g)$ and $dV_g$  is the volume element.

It is an area of interest in both number theory and harmonic analysis to study various quantitative behavior of the eigenfunctions restricted to certain smooth curve. Before stating our results, let us survey the recent developments in this area.
\subsection{Period Integrals}

Using the Kuznecov formula, Good \cite{Good} and Hejhal \cite{Hej} independently showed that if $\gamma_{per}$ is a periodic geodesic on a compact hyperbolic
surface $M$ then, uniformly in $\la$,
\begin{equation}\label{i.1}
\Bigl|\, \int_{\gamma_{per}} e_\la \, ds\, \Bigr| \le C_{\gamma_{per}}\|e_\la\|_{L^2(M)},
\end{equation}
with $ds$ denoting arc length measure on $\gamma_{per}$.

This result was generalized by Zelditch~\cite{ZelK}, who showed,  among many other things, that if $\la_j$ are the eigenvalues of $\sqrt{-\Delta_g}$  on
a compact Riemannian surface, and if $p_j(\gamma_{per})$ denote the period integrals in \eqref{i.1} 
for an orthonormal basis of eigenfunctions with eigenvalues $\la_j$,
then
\begin{equation}\label{Zel}\sum_{\la_j\le \la}|p_j(\gamma_{per})|^2 =c_{\gamma_{per}} \la +O(1),\end{equation}
which implies \eqref{i.1}.  Further work for hyperbolic surfaces giving more information about the lower order terms in terms
of geometric data for $\gamma_{per}$ was done by Pitt~\cite{Pitt}.  Since, by Weyl's Law, the number of eigenvalues (counting multiplicities) that are smaller than
$\la$ is $\sim\la^2$, this asymptotic formula \eqref{Zel} implies that, at least on average, one can do much better than \eqref{i.1}.  The problem
of improving this upper bound was raised and discussed in Pitt~\cite{Pitt} and Reznikov~\cite{Rez}.

In a paper of Chen and Sogge \cite{CSPer}, it was pointed out that \eqref{i.1} is sharp on compact
Riemannian surfaces of constant non-negative curvature.  Indeed, on the round sphere $S^2$, the integrals in \eqref{i.1} have unit size if
$\gamma_{per}$ is the equator and $e_\la$ is an $L^2$-normalized zonal spherical harmonic of even degree.  Also on the flat torus ${\mathbb T}^2$, for every
periodic geodesic, $\gamma_{per}$, one can find a sequence of eigenvalues $\la_k$ and eigenfunctions $e_{\la_k}$ so that
$e_{\la_k}\equiv 1$ on $\gamma_{per}$ and $\|e_{\la_k}\|_{L^2({\mathbb T}^2)}\approx 1$.  

Despite this, in \cite{CSPer}, it was shown that the period integrals in \eqref{i.1} are $o(1)$ as $\la\to \infty$ if $(M,g)$ has strictly negative
curvature.  The proof exploited the fact that, in this case, quadrilaterals always have their four interior angles summing to a value
strictly smaller than $2\pi$.   This allowed the authors to obtain $o(1)$ decay for period integrals using a stationary
phase argument involving reproducing kernels for the eigenfunctions. In a recent paper of Sogge, the author and Zhang \cite{gauss}, this method was further refined, and they managed to show that 
\begin{equation}\label{gauss}
\Bigl|\, \int_{\gamma_{per}} e_\la \, ds\, \Bigr| = O((\log\lambda)^{-\frac12}),
\end{equation}
under a weaker curvature assumption, by using the Gauss-Bonnet Theorem to get a quantitative version of the ideas used in \cite{CSPer}. Following similar routes, Wyman (\cite{emmett2}, \cite{emmett1}, \cite{emmett3}) generalized the results in \cite{CSPer} and \cite{gauss} to curves and submanifolds with certain interesting curvature assumptions; see also the recent work of Canzani, Galkowski and Toth \cite{toth} and Canzani, Galkowski \cite{canzani} for $o(1)$ bounds of averages over submanifolds under more general geometric assumptions.

\subsection{$L^p$-restrictions}
Another way of looking at the restriction of eigenfunctions to curves, is to estimate their $L^p(\gamma)$-restriction norms. The starting point of this direction of research is the  $L^p$-restriction bounds of Burq, G{\'e}rard and Tzvetkov \cite{BGT}, which says that
\begin{equation}\label{rest}\|e_\lambda\|_{L^p(\gamma)}\le C\lambda^{\sigma(d,p)}\|e_\lambda\|_{L^2(M)},
\end{equation}
where
\begin{eqnarray}\sigma(2,p)=
\begin{cases}
\frac14, &2\le p\le 4,
\cr \frac12-\frac1p, &4\le p\le\infty. \end{cases}
\end{eqnarray}
The special case when $p=2$ also follows from Theorem 1 in \cite{tataru}
Tataru  by considering an eigenfunction as a stationary solution of the wave equation, which then follows from classical estimates on Fourier integral operators with fold singularities by Greenleaf and Seeger \cite{Greenleaf}.

There has been considerable work towards improving \eqref{rest} when $\gamma$ is a geodesic, for Riemannian surface with non-positive curvature. B\'erard's sup-norm estimate \cite{Berard} provides natural improvements for large $p$. In \cite{chen}, Chen managed to improve over \eqref{rest} for all $p>4$ by a factor of $(\log\lambda)^{-\frac12}$:
\begin{equation}
\|e_\lambda\|_{L^p(\gamma)}\le C\frac{\lambda^{\frac12-\frac1p}}{(\log\lambda)^\frac12}\|e_\lambda\|_{L^2(M)}.
\end{equation}
Sogge and Zelditch \cite{SZStein} showed that one can improve \eqref{rest} for $2\le p<4$, in the sense that
\begin{equation}\label{oh}
\|e_\lambda\|_{L^p(\gamma)}=o(\lambda^\frac14),
\end{equation}
A few years later, Chen and Sogge \cite{ChenS} showed that the same conclusion can be drawn for $p=4$:
\begin{equation}\label{chensogge}
\|e_\lambda\|_{L^4(\gamma)}=o(\lambda^\frac14).
\end{equation}
\eqref{chensogge} is the first result to improve an estimate that is saturated by both zonal functions and highest weight spherical harmonics.
Recently, by using Toponogov's comparison theorem, Blair and Sogge \cite{BSTop} showed that it is possible to get log improvements for $L^2$-restriction:
\begin{equation}\label{top1}
\|e_\lambda\|_{L^2(\gamma)}\le C\frac{\lambda^\frac14}{(\log\lambda)^\frac12}\|e_\lambda\|_{L^2(M)},
\end{equation}
By following the same road map, further improvements were obtained by the author and Zhang \cite{xizhang} for the critical $L^4$-restriction estimate.
They showed that for manifold of non-positive curvature, 
	\begin{equation}\label{main}
\|e_\lambda\|_{L^4(\gamma)}\le C\lambda^\frac14(\log\log\lambda)^{-\frac18}.\end{equation}
And for surfaces with constant negative curvature
	\begin{equation}\label{eq7}\|e_\lambda\|_{L^4(\gamma)}\le C\lambda^\frac14(\log\lambda)^{-\frac12}\|e_\lambda\|_{L^2(M)}.\end{equation}
In a recent work of Blair \cite{Blair}, \eqref{eq7} has been generalized to all non-positively curved Riemannian surfaces. See also the recent work of Hezari \cite{hezari} in this direction. We also mention \cite{Bourgainef},
\cite{Skn} and \cite{BSAPDE} among further developments which reversed \eqref{rest}, in the sense that,
\begin{equation}\label{KN}
\|e_\la\|_{L^4(M)}\le C\la^{\frac1{8}+\eps}\sup_\gamma\|e_\la\|_{L^2{(\gamma)}}^\frac12,
\end{equation}
where the $\sup$ is taking over all unit length geodesics. See also \cite{bilinearKN} for a bilinear version of \eqref{KN}.

\subsection{Inner Product of Eigenfunctions over Smooth Curves}
Our goal in this paper, is to unify both the period integrals \eqref{i.1} and the $L^2$-restriction \eqref{rest} bounds by regarding them as different inner products of restricted eigenfunctions. We can think of \eqref{i.1} as the inner product of restricted eigenfunction $e_\lambda|_\gamma$ with the constant function $1|_\gamma$, which happens to be the eigenfunction associated with eigenvalue 0. On the other hand, we can think of \eqref{rest} as the inner product of $e_\la|_\gamma$ with itself. Therefore it is natural to bi-linearize this problem, and study the inner product of two eigenfunctions $e_\lambda$ and $e_\mu$ over a smooth closed curve $\gamma$, in the  form of
\begin{equation}\label{i.2}\int_{\gamma}e_\lambda \overline{e_\mu}\, ds.\end{equation}
Bilinear eigenfunction estimates arise naturally in the study of harmonic analysis and PDEs. For instance, in \cite{BGT'}, Burq, G{\'e}rard and Tzvetkov obtained general bilinear eigenfunction bounds over the whole manifold, in the sense that for $\mu\le\la$,
\begin{equation}\label{bilinear}\|e_\lambda e_\mu\|_{L^2(M)}\le C\mu^\frac{1}{4}.
\end{equation}
In contrast, Sogge's classical $L^p$ estimates \cite{Sef} together with H\"older inequality implies that \begin{equation}\label{holder}\|e_\lambda e_\mu\|_{L^2(M)}\le C\la^\frac18\mu^\frac{1}{8}.
\end{equation}
Clearly, \eqref{bilinear} provides a significant improvement over \eqref{holder}, which was then exploited in \cite{BGT'} to obtain sharp well-posedness results for the  Schr\"odinger equations on Zoll surfaces. Further refinements in the sense of \cite{Skn} can be found in \cite{bilinearKN}, see also the work of Guo, Han and Tacy \cite{han} which provides a full range of bilinear estimates.

We remark that the seemingly natural analog of \eqref{bilinear} for restricted eigenfunctions, i.e.,  $\|e_\lambda e_\mu\|_{L^2(\gamma)}$, cannot have any estimates better than the trivial bounds coming from \eqref{rest} and H\"older's inequality:
\begin{equation}\label{HH}\|e_\lambda e_\mu\|_{L^p(\gamma)}\lesssim\|e_\lambda e_\mu\|_{L^2(\gamma)}\le C\la^\frac14\mu^\frac14,\quad\text{for }1\le p\le2.\end{equation}
The reason is that if we take $Q_k$ and $Q_m$  to be two $L^2$-normalized highest weight spherical harmonics of degree $m$, and $k$ respectively, then on the equator $\gamma$, $|Q_k(\gamma(s))|\sim k^\frac14$, and $|Q_m(\gamma(s))|\sim m^\frac14$ for any $s$, and therefore,
$$\|Q_k Q_m\|_{L^p(\gamma)}\sim m^\frac14k^\frac14,\quad\text{for }1\le p\le\infty,$$
which indicates that \eqref{HH} is sharp for $1\le p\le 2$ on $S^2$.

Now let us turn our attention back to \eqref{i.2}. If we fix $\lambda$, and let $\mu\rightarrow\infty$, then the argument in \cite{ChenS} and \cite{gauss} shows that this integral will be uniformly bounded by $C_\mu$ for general Riemannian surfaces, and will decay like $C_\mu(\log\lambda)^{-\frac12} $ if $\gamma$ is a geodesic for negatively curved surfaces. It is then natural to ask, what happens if  $\lambda,\, \mu\rightarrow\infty$ simultaneously? First, for general surfaces, a simple observation is that if $\lambda=\mu$, this integral \eqref{i.2} resembles $\|e_\lambda\|_{L^2(\gamma)}^2$, which is $\gtrsim\lambda^\frac12$ on the sphere. On the other hand, \eqref{rest} implies that
\begin{equation}\left|\label{i.3}\int_{\gamma}e_\lambda \overline{e_\mu}\, ds\right|\le C\lambda^\frac14\mu^\frac14.\end{equation}
Then it is natural to ask whether one can obtain a better bound than the trivial estimate in \eqref{i.3} for $\la\neq\mu$. Such an improvement would show some degree of orthogonality for $e_\lambda|_\gamma$ and $e_\lambda|_\mu$, and is possible as shown by our main results.
\begin{theorem}\label{thm1}Let $(M,g)$ be a compact two-dimensional boundary-less Riemannian manifold. Let $\gamma$ be a smooth closed curve on $(M,g)$ parametrized by arc length, and let $|\gamma|$ denote
	its length. Then for $0\le\mu\le\la$,
	\begin{equation}\label{1/2}
	\Big|\int_{\gamma}e_\lambda \overline{e_\mu}\, ds\Big|\le C |\gamma|(1+\mu)^\frac12.
	\end{equation}
Moreover, given $0<c<1$, if we assume $0\le\frac{\mu}{\la}<c<1$ for some fixed constant $c$, then
	\begin{equation}\label{inner}
	\Big|\int_{\gamma}e_\lambda \overline{e_\mu}\, ds\Big|\le C |\gamma|(1+\mu)^\frac14,
	\end{equation}
	where the constant $C$ only depends on $M$ and $c$.
\end{theorem}

\eqref{1/2}  captures both the $O(1)$ bound of the period integral \eqref{i.1} (when $\mu=0$), and  the $L^2(\gamma)$-restriction bound of Burq, G{\'e}rard and Tzvetkov \cite{BGT} (when $\la=\mu$). Moreover, \eqref{1/2} and \eqref{inner} also share the flavor of the bilinear estimate \eqref{bilinear}, in the sense that the high frequency factor in the H\"older bounds disappears.  \eqref{inner} indicates that, even though the $L^2(\gamma)$-restriction norms of $e_\la$ and $e_\mu$ can be as large as $\la^\frac14$ and $\mu^\frac14$ respectively, their inner product over $\gamma$ cannot be as large, which implies the existence of some kind of orthogonality between $e_\la|_\gamma$ and $e_\mu|_\gamma$.
Note also that for $c\la\le\mu\le\la$, \eqref{1/2} is a simple consequence of \eqref{rest} and H\"older's inequality, so we only need to prove \eqref{inner}, which is stronger than \eqref{1/2} in the range $0\le\mu<c\lambda$.

By a standard partition of unity argument, we can see that Theorem \ref{thm1} follows from the following local results.
\begin{theorem}\label{thm2}
		Let $(M,g)$ be a compact two-dimensional boundary-less manifold. Let $\gamma$ be a smooth curve on $(M,g)$ parametrized by arc length. If $0\le\mu\le \la$, then for any test function  $\chi(s)\in C_0^\infty(-\frac12,\frac12)$ we have 
	\begin{equation}
\Big|\int_{\gamma}e_\lambda(\gamma(s))\, \overline{e_\mu(\gamma(s))}\,\chi(s)\, ds\Big|\le C (1+\mu)^\frac12.
\end{equation}
Moreover, if we assume $0\le\frac{\mu}{\la}<c<1$ for some fixed constant $c$, then
\begin{equation}
\Big|\int_{\gamma}e_\lambda(\gamma(s))\, \overline{e_\mu(\gamma(s))}\,\chi(s)\, ds\Big|\le C (1+\mu)^\frac14,
\end{equation}
	where the constant $C$ only depends on $M$ and $c$.
\end{theorem}

By essentially the same proof, Theorem \eqref{thm1} can be generalized to higher dimensions for inner product of eigenfunctions over hypersurfaces. 

\begin{theorem}\label{thm4}Let $(M,g)$ be a compact $n$-dimensional boundary-less Riemannian manifold. Let $H$ be a smooth closed hypersurface on $(M,g)$, and let $|H|$ denote
	its $(n-1)$-dimensional Hausdorff measure. Then for $0\le\mu\le\la$,
	\begin{equation}\label{1/2'}
	\Big|\int_{H}e_\lambda \overline{e_\mu}\, d\sigma\Big|\le C |H|(1+\mu)^\frac{n-1}{2}.
	\end{equation}
	Moreover, given $0<c<1$, if we assume $0\le\frac{\mu}{\la}<c<1$ for some fixed constant $c$, then
	\begin{equation}\label{inner'}
	\Big|\int_{H}e_\lambda \overline{e_\mu}\, d\sigma\Big|\le C |H|(1+\mu)^\frac14,
	\end{equation}
	where the constant $C$ only depends on $M$ and $c$.
\end{theorem}
Similar to Theorem \ref{thm1}, \eqref{1/2'}  captures both the $O(1)$ bound of the integral of eigenfunction over hypersurfaces of Zelditch \cite{ZelK} (when $\mu=0$), and  the $L^2(\gamma)$-restriction bound of Burq, G{\'e}rard and Tzvetkov \cite{BGT} and Tataru \cite{tataru} (when $\la=\mu$), which says that
\begin{equation}\label{rest'}
\|e_\lambda\|_{L^2(H)}\le C\lambda^\frac14.
\end{equation}
Note again that for $c\la\le\mu\le\la$, \eqref{1/2'} is a simple consequence of \eqref{rest'} and H\"older's inequality, so we only need to prove \eqref{inner'}, which is stronger than \eqref{1/2'} in the range $0\le\mu<c\lambda$.

\subsection{Generalized Periods.}
Another way of looking at the period integral \eqref{i.1} is to regard it as the 0-th order Fourier coefficient of $e_\la|_\gamma$. The general $\nu$-th order Fourier coefficients of $e_\la|_\gamma$ are called {\it generalized periods}. Generalized periods for hyperbolic surfaces naturally arise in the theory
of automorphic functions, and are of interest in their own right, thus have been studied recently by number theorists.  It is shown by Reznikov \cite{Rez} that on hyperbolic surfaces, if $\gamma$ is a periodic geodesic or a geodesic circle, the $\nu$-th order Fourier coefficients of $e_\lambda|_\gamma$ is uniformly bounded if $\nu\le c_\gamma\la$ for some constant $c_\gamma$ depending on $\gamma$. Another key insight, which is provided by the results in 
\cite{gauss}, is that the first $\nu$ Fourier coefficients of $e_\lambda(\gamma(t))$ as a function on $S^1$, are uniformly bounded by $C_\nu(\log\lambda)^{-\frac12}$.  In this spirit, as a byproduct of our argument, we generalize the result of Reznikov to any smooth closed curve over general Riemannian surfaces. 
\begin{theorem}\label{thm3}
	Let $\gamma$ be a smooth closed curve on $(M,g)$. Given $0<c<1$, if $0\le\frac{\nu}{\la}<c<1$,  then we have
	\begin{equation}\label{period}
	\Big|\int_{\gamma}e_\lambda(\gamma(s)) e^{-i\nu s}\, ds\Big|\le C |\gamma|,
	\end{equation}
	where the constant $C$ only depends on $(M,g)$ and $c$.
	In addition, if $\mu>c^{-1}\la$, then we have
		\begin{equation}
	\Big|\int_{\gamma}e_\lambda(\gamma(s)) e^{-i\nu s}\, ds\Big|\le C_N \nu^{-N},
	\end{equation}
		for all $N\in\mathbb N.$
\end{theorem}

{\bf Remark:} It is pointed out to the author by Professor Sogge that our proof for Theorem \ref{thm3} actually gives that
	\begin{equation}\label{period''}
\Big|\int_{\gamma}e_\lambda(\gamma(s)) e^{-i\nu s}\, ds\Big|\le C |\gamma|\|e_\la\|_{L^1(M)},
\end{equation}
which,  in the case of $\nu=0$, improves Hejhal-Good-Zelditch's bound \eqref{i.1} which had $\|e_\la\|_{L^2(M)}$ on the right hand side. In addition, this bound \eqref{period''} is sharp for both zonal and highest weight spherical harmonics.

Our Theorem \ref{thm3} not only generalizes the result of \cite{Rez} to general Riemannian surfaces and general closed curves, but also provides a stronger bound in the sense that our $c$ can be chosen to be any number strictly less than 1 and independent of $\gamma$, and our bound only depends on the length of $\gamma$.

Note that on the round sphere $S^2$, the highest weight spherical harmonics of degree $\nu$ restricted to the equator looks exactly like $\nu^\frac14e^{i\nu s}$, which highlights both the difference and connection between \eqref{inner} and \eqref{period}.

The broad strategy of the proof of our theorems will be a combination of the ideas in the work of Chen and Sogge \cite{CSPer} and the tools established by Burq, G{\'e}rard and Tzvetkov \cite{BGT} and Bourgain \cite{Bourgainef}.  We shall need to refine the stationary phase type
arguments used for period integrals in \cite{CSPer} and \cite{Bourgainef}, and make use of the additional oscillatory factor to get a further improvement when there is a gap between the frequencies $\la$ and $\mu$. 

This paper is organized as follows.  In the next section we shall state some basic facts that will be needed for our proof, including a standard reduction with a reproducing operator.  In \S 3, we prove Theorem \ref{thm3} by proving its local counterpart, this proof also serves as a model case for the proof of Theorem \ref{thm2}, which is much more involved.
In \S 4 we prove \ref{thm2} by invoking an argument of Bourgain. In \S 5 we show that both Theorem \ref{thm1} and \ref{thm3} are sharp. In section \S 6, we sketch the proof of Theorem \ref{thm4}. Finally, in section \S 7, we discuss some related problems.

From now on, we shall assume that the injectivity radius of $(M,g)$ is at least 10. We shall always use the letter $c$ to denote a positive constant that is strictly less than 1, and use the letter $C$ to denote various positive constants which only depends on $(M,g)$ and $c$.

\textit{Acknowledgments.}
The author would like to thank Professor Christopher Sogge, Allan Greenleaf and Alex Iosevich for their constant support and mentoring. In particular, the author want to thank Professor Sogge for his constructive comments, and it is a pleasure for the author to thank Professor Greenleaf for many helpful conversations, and for suggesting a related problem. The author also wants to thank Professor Xiaolong Han for some helpful suggestions.

 \section{A standard reduction}
 We shall work in the normal coordinates about a point $\gamma(0)$ on the curve $\gamma,$ such that $\dot\gamma(0)=(1,0)$, and thus we may assume that $g_{ij}(\gamma(0))=\delta_{ij}$.
 
If we take  a real-valued function $\rho\in {\mathcal S}(\R)$ satisfying
$$\rho(0)=1 \quad \text{and } \, \, \hat \rho(\tau)=0 , \quad |\tau|\not\in(c_3\eps_0,c_4\eps_0),$$
then we have $\rho(\la-\sqrt{-\Delta_g})e_\la=e_\la,$ and thus  $\rho(\la-\sqrt{-\Delta_g})$ is called a {\it reproducing} operator. By a standard wave operator trick, we have the following lemma.
\begin{lemma}[{\cite[Lemma 5.1.3]{SFIO}}]\label{513}For the above function $\rho \in \mathcal{S}(\mathbb{R})$, if we work in a small fixed coordinate patch $W$, and $d_g(x,y)$ denotes the Riemannian distance function, then there exists a smooth function $a_\la(x,y)$,
	supported in the set 
	\begin{equation*}
	U=\{(x,y): c_1 \eps_0\leq d_g(x,y)\leq
	c_2\eps_0\ll 1 \}\, 
	\end{equation*}and such that
	$$  |\partial^\alpha_{x}\partial^\beta_ya_\la(x,y)|\leq C_{\alpha\beta}\quad\text{for all  }\alpha,\,\beta\ge0,$$
	for which we have that, for $x\in W$,
\begin{equation}
\rho(\la-\sqrt{-\Delta_g})f(y)=\la^\frac12\int e^{i\la d_g(x,y)}a_\lambda(x,y)f(y)dy+R_\lambda f,
\end{equation}
where the reminder operator $R_\lambda$ is smoothing, and
\begin{equation}
\|R_\lambda\|_{L^1\rightarrow L^\infty}\le C_N\lambda^{-N},\quad\text{for any }N.
\end{equation}
\end{lemma}
We remark that $c_1, c_2, \eps_0$ only depend on the choice of $c_3, c_4$ for $\rho$. We can always choose $c_3$, $c_4$ to be arbitrarily small and $c_3/c_4$ arbitrarily close to 1, which allows us to assume $\epsilon_0$ to be arbitrarily small and both $c_1,\,c_2$ to be arbitrarily close to 1.

The proof of the lemma above is based on expressing the reproducing operator $\rho(\la-\sqrt{-\Delta_g})$ by a parametrix for the half wave
operator $
i\partial_{t}-\sqrt{-\Delta_g}
$.

Indeed, one may write
\begin{equation*}
\rho(\la-\sqrt{-{\Delta_g}})f
= 
\frac{1}{2\pi}\int\
e^{it\lambda}\,\hat{\rho}(t)\,
e^{-it\sqrt{-{\Delta_g}}}f
dt\,.
\end{equation*}
Assuming that the Fourier support of $\rho$ is small enough, 
one can use the H\"ormander parametrix to express $e^{-it\sqrt{-{\Delta_g}}}$ as a Pseudo-differential operator (See \cite[Chapter 4]{SFIO}). A stationary phase
argument (see \cite[Lemma 5.1.3]{SFIO}) then yields Lemma \ref{513}.

\section{Generalized periods: Proof of Theorem \ref{thm3}}We prove Theorem \ref{thm3} first, since it serves well as a simple model case.
It suffices to consider the integral
\begin{equation}\label{GP}
\int_\gamma\Big(\rho(\la-\sqrt{-\Delta_g})f\Big)(\gamma(s))e^{-i\nu s}\chi(s)\,ds,
\end{equation}
where $\chi(s)\in C_0^\infty((-\delta,\delta))$ is a fixed bump function which has small support, and $\nu<c\la$. If we can show that
\begin{equation}
\Big|\int_\gamma\Big(\rho(\la-\sqrt{-\Delta_g})f\Big)(\gamma(s))e^{-i\nu s}\chi(s)\,ds\Big|\le C\|f\|_{L^2(M)},
\end{equation}
by taking $f=e_\la$,  it follows that 
\begin{equation}
\Big|\int_\gamma e_\la\,e^{-i\nu s}\,\chi(s)\,ds\Big|\le C,
\end{equation}
and then by a standard partition of unity argument, Theorem \ref{thm3} follows.

Now, by Lemma \ref{513}, \eqref{GP} equals
\begin{equation}
\la^\frac12\iint e^{i\la d_g(\gamma(s),x)-i\nu s}a_\lambda(\gamma(s),x)\chi(s)f(x)\,dsdx+\int_\gamma R_\la f(\gamma(s))e^{-i\nu s}\chi(s)\,ds.
\end{equation}
Since the second term satisfies much better estimates, it remains to estimate
\begin{equation}
\la^\frac12\iint e^{i\la d_g(\gamma(s),x)-i\nu s}a_\lambda(\gamma(s),x)\chi(s)f(x)\,dsdx,
\end{equation}
where  $\supp \chi(s)\subset(-\delta,\delta)$, and the support of $a_\lambda(x,y)$  is contained in the set $\{(x,y):c_1\eps_0<d_g(x,y)<c_2\eps_0\}$. Here we fix $0<\delta\ll\eps_0\ll1$.  We have
\begin{equation}
|\partial_x^\alpha\partial_y^\beta a_\lambda(x,y)|\le C_{\alpha\beta},  \text{ for all }\alpha,\beta\ge0.
\end{equation}
 If we let $\alpha(s,x)=a_\lambda(\gamma(s),x)\chi(s),$ then it follows that
\begin{equation}
|\partial_x^\alpha\partial_s^\beta b(s,x)|\le C_{\alpha\beta}, \text{ for all }\alpha,\beta\ge0.
\end{equation}
Now we consider the kernel
\begin{equation}
K(s,x)=\int e^{i\la d_g(\gamma(s),x)-i\nu s}\alpha(s,x)\,ds,
\end{equation}
For the sake of simplicity, let $\epsilon=\frac\nu\la<c<1$, and denote 
\begin{equation}
\phi(s,x)=d_g(\gamma(s),x)-\epsilon s,
\end{equation}
then we may write
   \begin{equation}
K(s,x)=\int e^{i\la \phi(s,x)}\alpha(s,x)\,ds.
\end{equation}
Now consider a pair of points $(\gamma(s),z)\in\{(x,y):c_1\eps_0<d_g(x,y)<c_2\eps_0\}$. 
For such $\gamma(s), z$, we denote by $e_z$ the unit vector in $T_{\gamma(s)}M$ that generates the short
geodesic between $\gamma(s)$ and $z$, and write 
$\theta_z(s)  = \theta(\gamma(s), z)$ for the angle in $T_{\gamma(s)}M$ (measured
using the metric $g$) between the vector $e_z$ and the velocity vector $\gamma'(s)$ of $\gamma$ at $\gamma(s)$.
Since $(\gamma(s),z)\in\{(x,y):c_1\eps_0<d_g(x,y)<c_2\eps_0\}$, we know that $\gamma(s)$ and $z$ are at least  $\frac12 \eps_0$ distance apart, so $\theta_z(s)  = \theta(\gamma(s), z)$ is a smooth function of such $s$ and $z$. Then, it is clear that,
\begin{equation}
\frac{\partial \phi}{\partial s}(s,x)=\cos\theta_x(s)-\epsilon,
\end{equation}
which vanishes only if $\cos\theta(\gamma(s),x)=\epsilon$. At such a critical point $s_0$, we can see that
\begin{equation}\label{phi''}
\left|\frac{\partial^2 \phi}{\partial s^2}(s_0,x)\right|=\left|\sin\theta_x(s_0)\frac{d\theta_x}{d s}\right|,
\end{equation}
which has a non-trivial uniform lower bound independent of $\la$ and $\nu$, due to the fact that at $s_0$, $|\sin\theta_x(s_0)|>\sqrt{1-c^2}$, and $\left|\frac{d\theta_x}{d s}\right| =\frac1{d_g(\gamma(s),x)}\sin\theta_x+O(1)$. Therefore \eqref{phi''} is bounded below by a constant, if $\delta$ and $\eps_0$ are small enough.

Therefore, we can invoke the method of stationary phase to see that
   \begin{equation}
K(s,x)=O(\la^{-\frac12}),
\end{equation}
which shows that 
\begin{equation}
\la^\frac12\iint e^{i\la d_g(\gamma(s),x)-i\nu s}a_\lambda(\gamma(s),x)\chi(s)f(x)\,dsdx=O(1),
\end{equation}
finishing the proof of the first part of Theorem \ref{thm3}. 

For the second part,  we have $\epsilon=\frac\nu\la>c^{-1}>1$, and denote 
\begin{equation}
\tilde\phi(s,x)=\epsilon^{-1}d_g(\gamma(s),x)- s,
\end{equation}
then, in this case
\begin{equation}
K(s,x)=\int e^{i\nu \tilde\phi(s,x)}\alpha(s,x)\,ds,
\end{equation}
and
\begin{equation}
\left|\frac{\partial \tilde\phi}{\partial s}(s,x)\right|=|\epsilon^{-1}\cos\theta_x(s)-1|\ge 1-c
\end{equation}
Therefore, we may integrate by parts $N+1$ times to see that 
\begin{equation}
\la^\frac12\iint e^{i\la d_g(\gamma(s),x)-i\nu s}a_\lambda(\gamma(s),x)\chi(s)f(x)\,dsdx=O_N(\nu^{-N}),
\end{equation}
finishing the proof of Theorem \ref{thm3}
\section{Inner Product of Eigenfunctions over Curves: Proof of Theorem \ref{thm1} and \ref{thm2}}
It suffices to consider the integral
\begin{equation}\label{IP}
\int_\gamma\Big(\rho(\la-\sqrt{-\Delta_g})f\Big)(\gamma(s))\overline{\Big(\rho(\mu-\sqrt{-\Delta_g})g\Big)(\gamma(s))}\chi(s)\,ds,
\end{equation}
where $\chi(s)\in C_0^\infty((-\delta,\delta))$ is a fixed bump function which has small support, and $\nu<c\la$. If we can show that for $0\le\mu\le\la$,
\begin{multline}
\Big|\int_\gamma\Big(\rho(\la-\sqrt{-\Delta_g})f\Big)(\gamma(s))\overline{\Big(\rho(\mu-\sqrt{-\Delta_g})g\Big)(\gamma(s))}\chi(s)\,ds\Big|\\\le C\mu^{\frac12}\|f\|_{L^2(M)}\|g\|_{L^2(M)},
\end{multline}
and for $0\le\mu<c\la$, $0<c<1$,
\begin{multline}
\Big|\int_\gamma\Big(\rho(\la-\sqrt{-\Delta_g})f\Big)(\gamma(s))\overline{\Big(\rho(\mu-\sqrt{-\Delta_g})g\Big)(\gamma(s))}\chi(s)\,ds\Big|\\\le C\mu^{\frac14}\|f\|_{L^2(M)}\|g\|_{L^2(M)},
\end{multline}
then by taking $f=e_\la$, $g=e_\mu$, Theorem \ref{thm2} follows. Using a standard partition of unity argument, we can obtain Theorem \ref{thm1}.

Now, by Lemma \ref{513}, \eqref{IP} equals
\begin{multline}\label{4.4}
(\la\mu)^\frac12\iiint e^{i\la d_g(\gamma(s),x)-i\mu d_g(\gamma(s),y)}a_\lambda(\gamma(s),x)\overline{a_\mu(\gamma(s),y)}f(x)\overline{g(y)}\chi(s)\,dsdydx\\+\mu^\frac12\iint e^{-i\mu d_g(\gamma(s),y)}\overline{a_\mu(\gamma(s),y)}\overline{g(x)}R_\la f(\gamma(s))\chi(s)\,dsdy\\+\la^\frac12\iint e^{i\la d_g(\gamma(s),x)}a_\lambda(\gamma(s),x)f(x)\overline{R_\mu g(\gamma(s))}\chi(s)\,dsdx \\+\int R_\la f(\gamma(s))\overline{R_\mu g(\gamma(s))}\chi(s)\,ds.
\end{multline}
Clearly, the last term of \eqref{4.4} satisfies much better estimates. On the other hand, noticing that by Lemma \ref{513}, $R_\la f(\gamma(s)), R_\mu g(\gamma(s))\in C^\infty((-\delta,\delta))$ with bounded derivatives independent of $\la$ and $\mu$, we can see that by the proof in last section, the second and third term in \eqref{4.4} are bounded by $C\|f\|_{L^2(M)}\|g\|_{L^2(M)}$. Thus it remains to estimate the first term:
 \begin{equation}
 I=\la^\frac12\mu^\frac12\iiint e^{i\la d_g(\gamma(s),x)-i\mu d_g(\gamma(s),y)}a_\lambda(\gamma(s),x)\overline{a_\mu(\gamma(s),y)}f(x)\overline{g(y)}\chi(s)\,dsdydx,
 \end{equation}
 where  $\supp \chi(s)\subset(-\delta,\delta)$, and  support of both $a_\lambda(x,y)$ and $a_\mu(x,y)$ is contained in the set $\{(x,y):c_1\eps_0<d_g(x,y)<c_2\eps_0\}$. Here we fix $0<\delta\ll\eps_0\ll1$ which will be specified later in the proof.  We have
  \begin{equation}
|\partial_x^\alpha\partial_y^\beta a_\lambda(x,y)|\le C_{\alpha\beta} ,|\partial_x^\alpha\partial_y^\beta a_\mu(x,y)|\le C_{\alpha\beta}, \text{ for all }\alpha,\beta\ge0.
 \end{equation}
 
 If we let $b(s,x,y)=a_\lambda(\gamma(s),x)\overline{a_\mu(\gamma(s),y)}\chi(s),$ then it is clear that
   \begin{equation}
|\partial_x^\alpha\partial_y^\beta \partial_s^\kappa b(s,x,y)|\le C_{\alpha\beta \kappa}, \text{ for all }\alpha,\beta, \kappa\ge0.
 \end{equation}
Now let us consider the kernel
   \begin{equation}
K_0(x,y)=\int e^{i\la d_g(\gamma(s),x)-i\mu d_g(\gamma(s),y)}b(s,x,y)\,ds.
\end{equation}
We can write $I$ as
   \begin{equation}\label{IK}
I= \la^\frac12\mu^\frac12\iint K_0(x,y) f(x)\overline{g(y)}\,dydx,
   \end{equation}
As before, let $0<\epsilon=\frac\mu\la<c<1$, and denote 
   \begin{equation}
\Phi(s,x,y)=d_g(\gamma(s),x)-\epsilon d_g(\gamma(s),y),
\end{equation}
then we have
   \begin{equation}
K_0(x,y)=\int e^{i\la \Phi(s,x,y)}b(s,x,y)\,ds.
\end{equation}

Now consider a pair of points $(\gamma(s),z)\in\{(x,y):c_1\eps_0<d_g(x,y)<c_2\eps_0\}$. 
For such $\gamma(s)$ and $z$, we denote by $e_z$ the unit vector in $T_{\gamma(s)}M$ that generates the shortest
geodesic between $\gamma(s)$ and $z$, and write 
$\theta_z(s)  = \theta(\gamma(s), z)$ for the angle in $T_{\gamma(s)}M$ (measured
using the metric $g$) between the vector $e_z$ and the velocity vector $\gamma'(s)$ of $\gamma$ at $\gamma(s)$.
Since $(\gamma(s),z)\in\{(x,y):c_1\eps_0<d_g(x,y)<c_2\eps_0\}$, we know that $\gamma(s)$ and $z$ are at least  $c_1\eps_0$ distance apart, so $\theta_z(s)  = \theta(\gamma(s), z)$ is a smooth function of such $s$ and $z$. Then, it is clear that,
\begin{equation}
\frac{\partial \Phi}{\partial s}(s,x,y)=\cos\theta(\gamma(s),x)-\epsilon\cos\theta(\gamma(s),y).
\end{equation}

 \begin{figure}[!ht]
  \centering
    \includegraphics[width=0.8\textwidth]{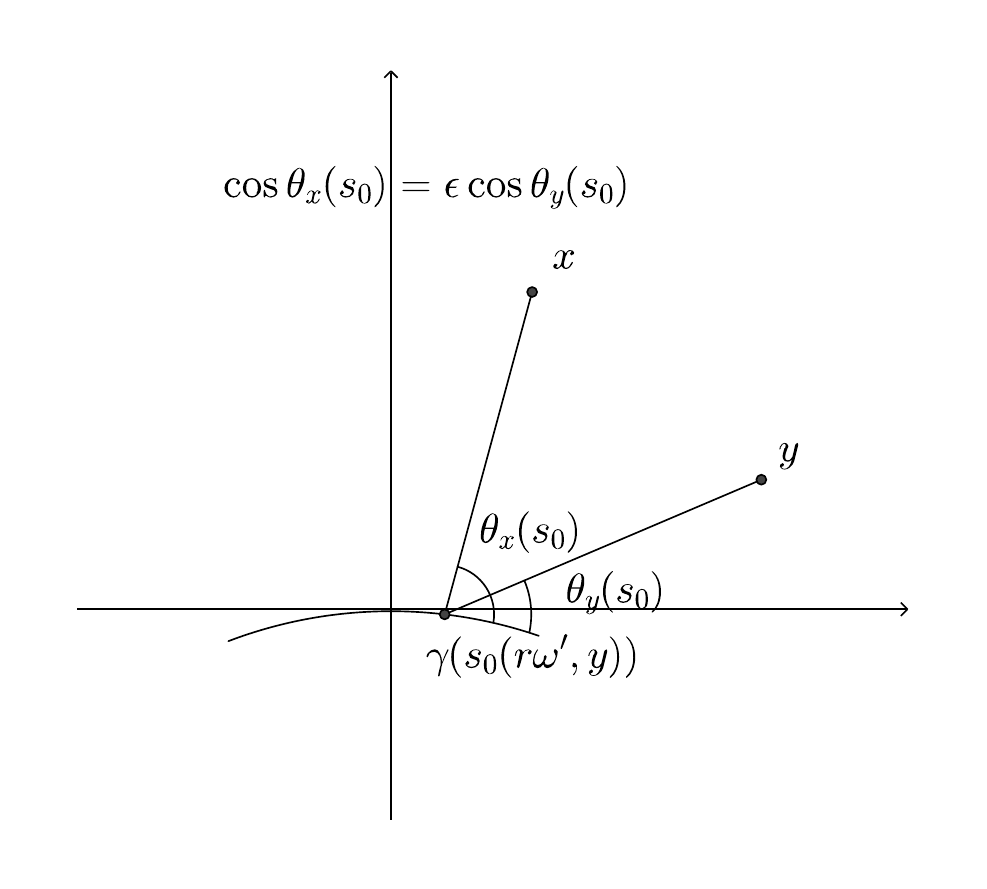}
      \caption{}
      \label{fig1}
\end{figure}
Without loss of generality, we assume in addition that for a pair of points $x,y\in\{z:c_1\eps_0<d_g(\gamma(0),z)<c_2\eps_0\}$, $s_0=s_0(x,y)$ is the only point in $(-2\delta,2\delta)$ such that $\partial_s\Phi(s,x,y)$ vanishes. (See Figure 1.) Indeed, if such a critical point does not exist, integration by parts yields $I=O_N(\lambda^{-N})$. It is then clear that, in this case, $s_0(x,y)$ is a smooth function of $x$ and $y$. On the other hand, at $s_0(x,y)$, we have
\begin{equation}\label{theeq}
\frac{\partial \Phi}{\partial s}(s_0(x,y),x,y)=\cos\theta_x(s_0(x,y))-\epsilon\cos\theta_y(s_0(x,y)=0,
\end{equation}
and 
\begin{equation}
\frac{\partial^2 \Phi}{\partial s^2}(s_0(x,y),x,y)=-\sin\theta_x\frac{d\theta_x}{d s}(s_0(x,y))+\epsilon\sin\theta_y\frac{d\theta_y}{d s}(s_0(x,y)).
\end{equation}

As we are assuming $g_{ij}(\gamma(0))=\delta_{ij}$, and $\delta$ and $\eps_0$ is sufficiently small, everything is close to being Euclidean, it is easy to see that
\begin{equation}\label{Eu}\left|\frac{d\theta_x}{d s}\right| \approx\footnote{Here we are off by  O(1) due to the (possible) geodesic curvature of $\gamma$, however, it is negligible since $d_g(\gamma(s),x)$ is assumed to be sufficiently small. }\frac1{d_g(\gamma(s),x)}\sin\theta_x,\end{equation}
and therefore,
\[\left|\frac{d\theta_x}{d s}(s_0(x,y))\right| \ge (c_2\eps_0)^{-1} \sin\theta_x,\ \left|\frac{d\theta_y}{d s}(s_0(x,y))\right| \le (c_1\eps_0)^{-1} \sin\theta_y.\]
Thus
\begin{align*}
\left|\frac{\partial^2 \Phi}{\partial s^2}(s_0(x,y),x,y)\right|&\ge \eps_0^{-1}(c_1   \sin^2\theta_x- c_2\epsilon  \sin^2\theta_y)
\\&=\eps_0^{-1}(c_2^{-1}(1-\cos^2\theta_x)- c_1^{-1}\epsilon (1-\cos^2\theta_y))\\
&=\eps_0^{-1}[(c_2^{-1}-c_1^{-1}\epsilon)-(c_2^{-1}\cos^2\theta_x-c_1^{-1}\epsilon\cos^2\theta_y)]
\\&=\eps_0^{-1}[(c_2^{-1}-c_1^{-1}\epsilon)-(c_2^{-1}\epsilon^2\cos^2\theta_y-c_1^{-1}\epsilon\cos^2\theta_y)]
\\&=\eps_0^{-1}[(c_2^{-1}-c_1^{-1}\epsilon)-\epsilon\cos^2\theta_y(c_2^{-1}\epsilon-c_1^{-1})]
\\&\ge\eps_0^{-1}(c_2^{-1}-c_1^{-1}c).
\end{align*}
Notice that for any given $0<c<1$, we can choose $0<c_1<c_2$ such that $c_1>c_2c$, hence $c_2^{-1}-c_1^{-1}c>c_0>0$, which implies that 
\begin{align}
\left|\frac{\partial^2 \Phi}{\partial s^2}(s_0(x,y),x,y)\right|>c_0\eps_0^{-1}.
\end{align}

Now we are in a position to invoke the stationary phase method, which yields
   \begin{equation}\label{stationary}
K_0(x,y)=\lambda^{-\frac{1}{2}}e^{i\la\Phi(s_0(x,y),x,y)}\tilde b(x,y),
\end{equation}
where 
\begin{equation}
 |\partial_x^\alpha\partial_y^\beta \tilde b(x,y)|\le C_{\alpha\beta}, \text{ for all }\alpha,\beta\ge0.
 \end{equation}
Then it is easy to see  \eqref{IK} and \eqref{stationary} that $|I|\le\mu^\frac12$. In fact,  further improvements can be obtained by looking at $I$ more closely and take advantage of the oscillating factor in \eqref{stationary}.

From now on, our proof will be somewhat similar to the proof  \cite[(1.1)]{BGT}, which used an argument that originated in the work on the local smoothing of Fourier integral operators by Mockenhaupt, Seeger and Sogge \cite{Mock}, and the work on Fourier integral operators with fold singularities by Greenleaf and Seeger \cite{Greenleaf}. Now if we write $x$ in polar coordinates, $x=r\omega$, then if we freeze the $r$ variable, by Cauchy-Schwarz inequality, we see that
\begin{align*}
   |I|^2&=\mu\left|\iiint e^{i\la\Phi(s_0(r\omega,y),r\omega,y)}\tilde b(r\omega,y) f(r\omega)\overline{g(y)}\,dydrd\omega\right|^2\\
   &\le\mu\left(\int\left|\iint e^{i\la\Phi(s_0(r\omega,y),r\omega,y)}\tilde b(r\omega,y) f(r\omega)\overline{g(y)}\,dyd\omega\right|dr\right)^2\\
   &\le\mu\int\left|\iint e^{i\la\Phi(s_0(r\omega,y),r\omega,y)}\tilde b(r\omega,y) f(r\omega)\overline{g(y)}\,dyd\omega\right|^2dr.
   \end{align*}
It then suffices to bound    
   \begin{equation}
II=\left|\iint e^{i\la\Phi(s_0(r\omega,y),r\omega,y)}\tilde b(r\omega,y) f(r\omega)\overline{g(y)}\,dyd\omega\right|^2,\end{equation}
uniformly in $r$. In fact
   \begin{multline}
   \quad II\le\mu\|g\|_{L^2(M)}^2\int\left|\int e^{i\la\Phi(s_0(r\omega,y),r\omega,y))} \tilde b(r\omega,y) f(r\omega)\,d\omega\right|^2dy\\
   =\mu\|g\|_{L^2(M)}^2\iiint e^{i\la[\Phi(s_0(r\omega,y),r\omega,y)-\Phi(s_0(r\omega',y),r\omega',y)]}\overline{\tilde b(r\omega',y)}\tilde b(r\omega,y)\\\quad\quad\quad\quad \times f(r\omega)\overline{f(r\omega')}\,d\omega d\omega'dy.
   \end{multline}
We claim that
   \begin{eqnarray}
\left|\iiint e^{i\la[\Phi(s_0(r\omega,y),r\omega,y)-\Phi(s_0(r\omega',y),r\omega',y)]}\tilde b(r\omega,y)\overline{\tilde b(r\omega',y)} f(r\omega)\overline{f(r\omega')}\,d\omega d\omega'dy\right|\nonumber\\\le C\mu^{-\frac{1}{2}}\|f\|^2_{L^2(d\omega)},\label{last}
\end{eqnarray}
   which implies $|I|\le C\mu^\frac14$, and thus implies  Theorem \ref{thm1} and \ref{thm2}.
   
In fact, if we denote
\[B_r(\omega,\omega',y)=\tilde b(r\omega,y)\overline{\tilde b(r\omega',y)},\]
and let
   \begin{equation}\label{3}
K_r(\omega,\omega')=\int e^{i\la[\Phi(s_0(r\omega,y),r\omega,y)-\Phi(s_0(r\omega',y),r\omega',y)]}B_r(\omega,\omega',y)\,dy.
   \end{equation}
Moreover, since from now on we are interested in oscillations at frequency $\mu$, it is convenient to write
\begin{align*}\label{4}
 \Psi_r(\omega,\omega',y)&=\epsilon^{-1}[\Phi(s_0(r\omega,y),r\omega,y)-\Phi(s_0(r\omega',y),r\omega',y)]\\&=\epsilon^{-1}d_g(\gamma(s_0),r\omega)- d_g(\gamma(s_0),y)-[\epsilon^{-1}d_g(\gamma(s_0),r\omega')- d_g(\gamma(s_0),y)].
\end{align*}
Then $K_r(\omega,\omega')$ takes the following form

   \begin{equation}\label{5}
   K_r(\omega,\omega')=\int e^{i\mu\Psi_r(\omega,\omega',y)}B_r(\omega,\omega',y)\,dy.
   \end{equation}
We shall need the following lemmas.
   \begin{lemma}[{\cite[Lemma 3.1]{BGT}}]\label{Gauss'} Using geodesic polar coordinates, write $y=\exp_0{r\omega}$, with $r>0$, $\omega\in S^1$. Assuming $g(0)=Id$, we have
	\begin{equation}\nabla_x d_g(x,y)|_{x=0}=\omega.
	\end{equation}
\end{lemma}
Lemma \ref{Gauss'} is a simple consequence of Gauss' Lemma, and we shall omit the proof here.
 \begin{lemma}\label{6} For the kernel $K_r(\omega,\omega')$, there exists a constant $C$ such that
   \begin{equation}|K_r(\omega,\omega')|\le C(1+\mu|\omega-\omega'|)^{-\frac12}
   \end{equation}
   \end{lemma}
Let us postpone the proof of  Lemma \ref{6}, and see that how to use it to  prove \eqref{last}, finishing the proof of our main theorem.
\begin{proof}[Proof of \eqref{last}]

   \begin{align*}\label{7}
&\ \ \left|\iiint e^{i\la[\Phi(s_0(r\omega,y),r\omega,y)-\Phi(s_0(r\omega',y),r\omega',y)]}\tilde b(r\omega,y)\overline{\tilde b(r\omega',y)} f(r\omega)\overline{f(r\omega')}\,d\omega d\omega'dy\right|\\
&=\left|\iint K_r(\omega,\omega')f(r\omega)\overline{f(r\omega')}\,d\omega d\omega' \right|\\
&\le C\|f\|_{L^2(d\omega)}\Big\|\int (1+\mu|\omega-\omega'|)^{-\frac12} f(r\omega')d\omega' \Big\|_{L^2(d\omega)}\\
&\le C\|f\|_{L^2(d\omega)}^2\Big\|\int (1+\mu|\omega-\omega'|)^{-\frac12} d\omega' \Big\|_{L^1(d\omega)}\\
&\le C\mu^{-\frac12}\|f\|_{L^2(d\omega)}^2.
\end{align*}
Here on the second to the last line we used Hardy-Littlewood-Sobolev inequality.
	\end{proof}
Now it remains to prove Lemma \ref{6}. We shall employ the strategy in \cite{Bourgainef}.
\begin{proof}[Proof of Lemma \ref{6}]
We need to estimate
   \begin{equation}\label{8}
K_r(\omega,\omega')=\int e^{i\mu\Psi_r(\omega,\omega',y)}B_r(\omega,\omega',y)\,dy.
\end{equation}

Recall that 
\begin{equation}\label{4}
\quad\Psi_r(\omega,\omega',y)=\epsilon^{-1}[\Phi(s_0(r\omega,y),r\omega,y)-\Phi(s_0(r\omega',y),r\omega',y)],
\end{equation}
and
\begin{equation}
\epsilon^{-1}\Phi(s_0(r\omega,y),r\omega,y)=\epsilon^{-1}d_g(\gamma(s_0(r\omega.y)),r\omega)- d_g(\gamma(s_0(r\omega,y)),y).
\end{equation}
Since $s_0(r\omega,y)$ solves the equation $\frac{\partial \Phi}{\partial s}(s_0(r\omega,y),r\omega,y)=0$, it follows that
\begin{equation}
\epsilon^{-1}\nabla_y\Phi(s_0(r\omega,y),r\omega,y)
=-\nabla_yd_g(\gamma(s_0),y),
\end{equation}
thus
\begin{equation}
\nabla_y\Psi_r(\omega,\omega',y)=-\nabla_yd_g(\gamma(s_0(r\omega,y)),y)+\nabla_yd_g(\gamma(s_0(r\omega',y)),y).
\end{equation}
 
\begin{figure}[!ht]
  \centering
    \includegraphics[width=1\textwidth]{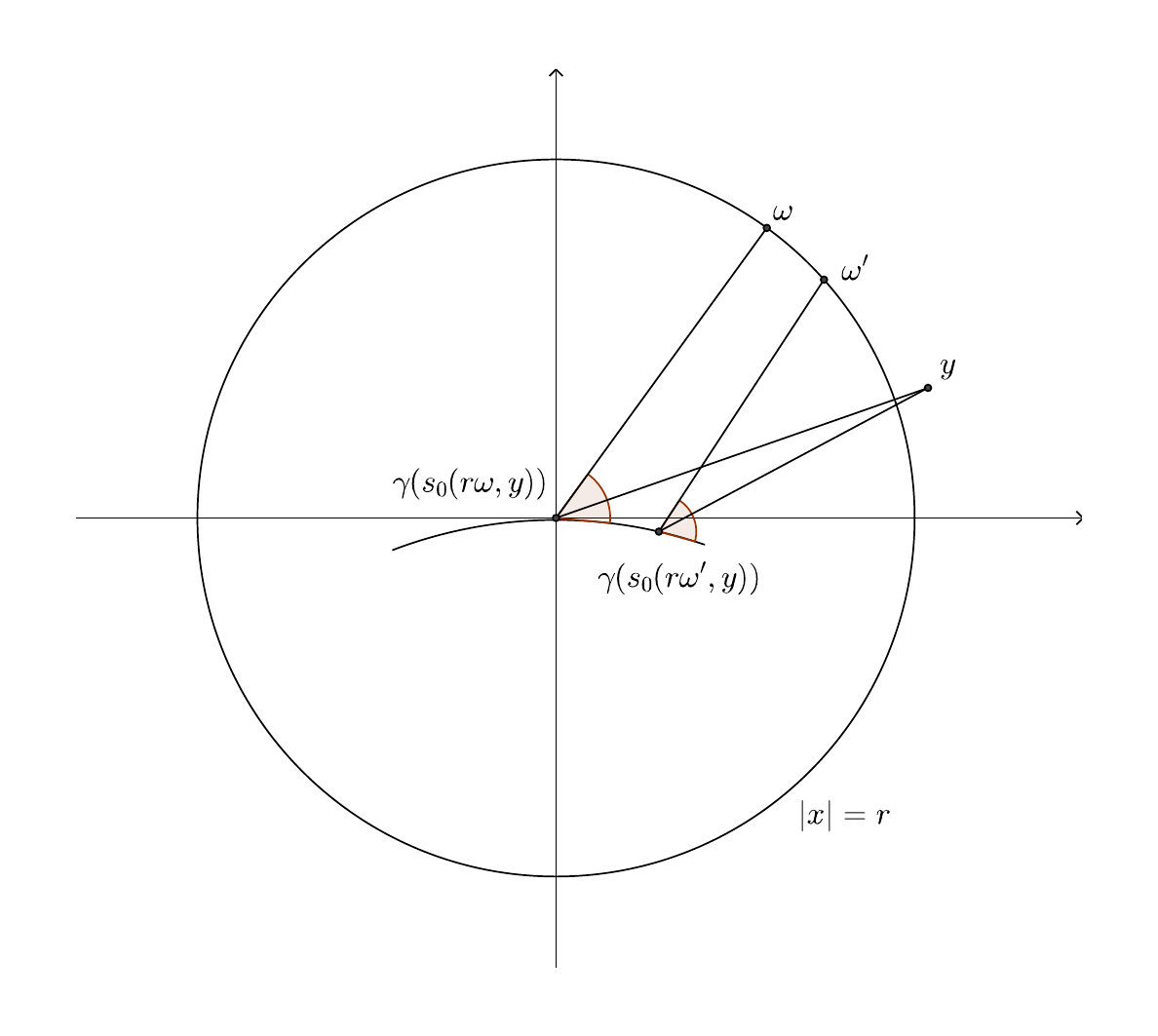}
      \caption{}
      \label{fig2}
\end{figure}
Applying Lemma \ref{Gauss'} with $0$ replaced by $y$, we see that (see Figure 2)
\begin{align}
|\nabla_y\Psi_r(\omega,\omega',y)|&=|\nabla_yd_g(\gamma(s_0(r\omega,y)),y)-\nabla_yd_g(\gamma(s_0(r\omega',y)),y)|\label{2d}\\
&=|\exp_y(\gamma(s_0(r\omega,y)))-\exp_y(\gamma(s_0(r\omega',y))|\nonumber\\
&\sim |\gamma(s_0(r\omega,y))-\gamma(s_0(r\omega',y))||y_2|\nonumber\\
&\gtrsim |\omega-\omega'||y_2|\nonumber,
\end{align}
here $y=(y_1,y_2)$, and we are using the fact that $\frac{\partial s_0}{\partial \omega}$ is bounded away from $0$, since $\omega$ is bounded away from 0 and $\pi$. Indeed, by differentiating both sides of \eqref{theeq} with respect to $\omega$, we see that
\[\sin\theta_{r\omega}(s_0)\frac{d\theta}{d\omega}(s_0,r\omega)+\sin\theta_{r\omega}(s_0)\frac{d\theta_{r\omega}}{ds_0}\frac{ds_0}{d\omega}+\epsilon\sin\theta_{y}(s_0)\frac{d\theta_y}{ds_0}\frac{ds_0}{d\omega}=0.\]
Recall that by \eqref{Eu}, $$\left|\frac{d\theta_x}{d s}\right| \approx\frac1{d_g(\gamma(s),x)}\sin\theta_x.$$ Therefore, by choosing $\delta, \eps_0$ small enough and $c_1, c_2$ sufficiently close to 1, we have
\[\Big|\frac{ds_0}{d\omega}\Big|\approx\frac1r|\sin\omega(\sin^2\theta_{r\omega}(s_0)\pm\eps\sin^2\theta_y(s_0))^{-1}|\ge\frac1r \sin\omega(1-\epsilon+(\epsilon^{-1}-1)\cos^2\theta_x),\]
which is clearly bounded away from 0 if $\omega$ is bounded away from $0, \pi$, and $\epsilon<c<1$.
 
Now if we let $\Omega_k$ denotes the set of $y$ such that $\dist(y,\{(z_1,z_2):z_2=0\})\sim 2^{-k} $, that is
\begin{equation}\label{omega_k}
\Omega_k=\{(y_1,y_2)\in U:2^{-k-1}<|y_2|\le 2^{-k}\},
\end{equation}
then integration by parts once yields the following bound:
   \begin{equation}\label{9}
\Big|\int_{\Omega_k} e^{i\mu\Psi_r(\omega,\omega',y)}B_r(\omega,\omega',y)\,dy\Big|\le C\min\Big\{\frac{2^k}{\mu|\omega-\omega'|},2^{-k}\Big\}.
\end{equation}
Therefore,
   \begin{align*}\label{9}
|K_r(\omega,\omega')|&\le\sum_{k=1}^\infty\min\Big\{\frac{2^k}{\mu|\omega-\omega'|},2^{-k}\Big\}\\
&=\sum_{k=1}^{[\log_4(\mu|\omega-\omega'|)]}\frac{2^k}{\mu|\omega-\omega'|}+\sum_{k=[\log_4(\mu|\omega-\omega'|)]+1}^{\infty}2^{-k}\\
&\le C(1+\mu|\omega-\omega'|)^{-\frac12}.
\end{align*}		\end{proof}

\section{Sharpness of Theorem \ref{thm1} and \ref{thm3}}
In this section, we shall discuss the sharpness of our results. We shall show that all of our results are sharp on $S^2$, and Theorem \ref{thm3} is also sharp on the flat torus $\mathbb T^2$. 
\subsection{The Sphere $S^2$}  To fix our notation, we consider $M = {S}^2$ with
the standard parametrization:
\[
(x, y, z) = ( \sin \theta \cos \varphi, \sin \theta \sin \varphi, \cos
\theta),
\]
where $0 \leq \theta \leq \pi$ is the angle from the north pole $(0,0,1)$ and $0
\leq \varphi \leq 2 \pi$ is the angle in the $xy$-plane measured
from the $x$-axis.  The induced metric is the usual spherical
metric:
\[
g = d \theta^2 + \sin^2 \theta d \varphi^2,
\]
and the volume form is
\[
dV_g = \sin \theta d \theta d \varphi.
\]
The Laplacian is the usual angular part of the polar Laplacian:
\[
-\Delta_g = -\frac{1}{\sin \theta} \partial_\theta \sin \varphi \partial_\theta -
\frac{1}{\sin^2 \theta } \partial_\varphi^2.
\]
The eigenfunctions for $-\Delta_g$ are homogeneous harmonic
polynomials in $\mathbb R^3$ restricted to $M$, with a standard $L^2(S^2)$-orthonormal basis $Y_\ell^m(\theta,\varphi)$ of the form:
\begin{equation}Y_\ell^m(\theta , \varphi ) =  \sqrt{(2\ell+1) \frac{(\ell-m)!}{ (\ell+m)!}} \, P_\ell^m ( \cos{\theta} )\, e^{i m \varphi},\end{equation}
where $P_\ell^m$ is the associated Legendre polynomials.
As an eigenfunction, $Y_\ell^m$ has eigenvalue $\sqrt{\ell(\ell+1)}$, and thus satisfies
\[
-\Delta_g Y_\ell^m = \ell(\ell+1) Y_\ell^m.
\]
For the special case when $m=\ell$, $Y_m^m(\theta,\varphi)$ is called the highest weight spherical harmonics, and is equal to $N_m e^{im\varphi}$ on the equator $\theta=\pi/2$, with a normalizing factor $N_m\sim m^\frac14$.

On the other hand, if both $\ell$ and $m$ are even and $m< c\ell$ for some $0<c<1$, then  the restriction of 
$Y_\ell^m(\theta,\varphi)$ to the equator is of the form
\begin{equation}\label{sh}Y_\ell^m(\pi/2 , \varphi ) =  \sqrt{{(2\ell+1) }{\frac{(\ell-m)!}{ (\ell+m)!}}} \, P_\ell^m (0)\, e^{i m \varphi },\end{equation}
where $P_\ell^m(0)$ can be computed inductively. One has
\[P_\ell^m(0)=\frac{(\ell-1)!!}{(\ell)!!}\frac{(\ell+m)!!}{(\ell-m)!!},\]
where $({}\cdot{})!!$ is the standard double factorial. A simple calculation shows that 
	\[\frac{(\ell-1)!!}{(\ell)!!}\sim \ell^{-\frac12},\quad\text{as }\ell\rightarrow\infty.\]
Therefore,  \eqref{sh} can be written as,
\begin{equation}Y_\ell^m( \pi/2 , \varphi ) =  \sqrt{{(2\ell+1) }{\frac{(\ell-m)!}{ (\ell+m)!}}} \,\frac{(\ell-1)!!}{(\ell)!!}\frac{(\ell+m)!!}{(\ell-m)!!}\, e^{i m \varphi },\end{equation}

Now the inner product of $Y_\ell^m(\pi/2,\varphi)$ with $Y_m^m(\pi/2,\varphi)$ over the equator $\gamma$ is the following:
\begin{align*}\int_\gamma Y_\ell^m(\pi/2,\varphi) 
\overline{Y_m^m(\pi/2,\varphi)}\,d\varphi&=N_m\, \sqrt{{(2\ell+1) }{\frac{(\ell-m)!}{ (\ell+m)!}}} \,\frac{(\ell-1)!!}{(\ell)!!}\frac{(\ell+m)!!}{(\ell-m)!!}\\
&\sim m^\frac14 \sqrt{\frac{(\ell-m)!}{ (\ell+m)!}} \frac{(\ell+m)!!}{(\ell-m)!!}\\
&=m^\frac14 \sqrt{\frac{(\ell-m)!}{ (\ell+m)!} \frac{[{(\ell+m)!!}]^2}{[{(\ell-m)!!}]^2}},
\end{align*}
where the term under the square root is
\begin{align*}{\frac{(\ell-m)!}{ (\ell+m)!}} \frac{[{(\ell+m)!!}]^2}{[{(\ell-m)!!}]^2}&=\prod\limits_{k=1}^{m}\frac{(\ell-m+2k)^2}{(\ell-m+2k-1)(\ell-m+2k)}\\
&=\prod\limits_{k=1}^{m}\frac{\ell-m+2k}{\ell-m+2k-1}>1.	\end{align*}
Indeed, the above quantity will also be $\lesssim_c1$ if $m<c\ell$, since 
\begin{equation*}\prod\limits_{k=1}^{m}\frac{\ell-m+2k}{\ell-m+2k-1}=\frac{\ell+m}{\ell-m+1}\prod\limits_{k=2}^{m}\frac{\ell-m+2k-2}{\ell-m+2k-1}<\frac{\ell+m}{\ell-m+1}<\frac2{1-c}.	\end{equation*}
and left hand side goes to $\ell^\frac14$ as $c\rightarrow 1$.
To sum up, we just proved that
\begin{equation}\Big|\int_\gamma Y_\ell^m(\pi/2,\varphi) 
\overline{Y_m^m(\pi/2,\varphi)}\,d\varphi\Big|\sim m^\frac14,\end{equation}
if $m<c\ell$, which shows that both bounds in Theorem \ref{thm1} are saturated on the sphere.
Similarly
\begin{equation}\Big|\int_\gamma Y_\ell^m(\pi/2,\varphi) e^{-im\varphi}\,d\varphi\Big|\sim 1,\end{equation}
if  $m<c\ell$, which indicates that Theorem \ref{thm3} is also sharp on the sphere.

\subsection{The Flat Torus $\mathbb T^2$}
It should be pointed out that estimates in Theorem \ref{thm1} are not optimal in the case of the
flat torus $\mathbb T^2$. In this case we have a strong improvement
of the Weyl bound on the $L^{\infty}$ norm of the eigenfunctions, in the sense that every eigenfunction is essentially bounded in this case. 
Indeed, in this case, we have the following explicit representations of
eigenfunctions:
\begin{equation}\label{rep}
e_{\lambda}(x,y)=\sum_{m^2+n^2=\lambda^2}a_{mn}e^{i(mx+ny)},
\end{equation}
thus by the Plancherel identity and the estimate for lattice points:
$$\#\big( (m,n)\in \mathbb Z^2\, : \,
m^2+n^2=\lambda^2\big)\leq C_{\varepsilon}(1+\lambda)^{\varepsilon},\quad \varepsilon>0,$$
we see that
\begin{equation}
\|e_{\lambda}\|_{L^{\infty}(\mathbb T^2)}\leq
C_{\varepsilon}(1+\lambda)^{\varepsilon}\|e_{\lambda}\|_{L^2(\mathbb T^2)}\, .
\end{equation}
Therefore, for every curve $\gamma$ on $\mathbb T^2$, we have the bound
\begin{equation}
\Big|\int_{\gamma}e_\la\,\overline{e_\mu}\,ds\Big|\leq C_{\varepsilon}(1+\lambda\mu)^{\varepsilon}
\|e_{\lambda}\|_{L^2(\mathbb T^2)}\|e_{\mu}\|_{L^2(\mathbb T^2)}\,,
\end{equation}
which is in general better than Theorem \ref{thm1}. It would be interesting to find a sharp bound in this special case.

On the other hand,  the generalized periods estimate in Theorem \ref{thm3} is sharp on $\mathbb T^2$. if we consider the restriction of the eigenfunction $e^{i\la x}$ to the geodesic $y=\tan\theta x\mod 2\pi$. Then for suitable $\la$ and $\theta$, if we set $\nu=\la\cos\theta$, then
\begin{equation}
\Big|\int_{\gamma}e^{i\la\cos\theta s}\,e^{-i\nu s}\,ds\Big|\sim 1\,,
\end{equation} Therefore Theorem \ref{thm3} is also  sharp on the torus $\mathbb T^2$.

\section{Generalization to Higher Dimensions: Inner product of Eigenfunctions over Hypersurfaces}
In this section, we sketch the proof of Theorem \ref{thm4}, which is parallel to the proof of Theorem \ref{thm1} for the 2-dimensional case.  We shall work in the normal coordinates about a point $H(0)$ on the hypersurface $H,$ such that the normal vector at $H(0)$ is equal to $e_1=(1,\vec 0)$. We shall parametrize $H$ naturally by the variable $z\in \mathbb R^{d-1}$ which is the projection of $H(z)$ onto the hyperplane $\{(x_1,x_2,\ldots,x_n):x_1=0\}$.

We shall need the following higher dimensional analog of Lemma \ref{513} which is also in \cite[Lemma 5.1.3]{SFIO}):
	\begin{equation}
	\rho(\la-\sqrt{-\Delta_g})(x,y)=\la^\frac{d-1}2e^{i\la d_g(x,y)}a_\lambda(x,y),
	\end{equation}
modulo a smoothing operator. The amplitude $a_\la(x,y)$ is supported in the set
	\begin{equation*}
U=\{(x,y): c_1 \eps_0\leq d_g(x,y)\leq
c_2\eps_0\ll 1 \}\, 
\end{equation*} and satisfies that
$$  |\partial^\alpha_{x}\partial^\beta_ya_\la(x,y)|\le C_{\alpha\beta}\quad\text{for all  }\alpha,\,\beta\ge0.$$ 
If we denote $b(z,x,y)=a_\lambda(H(z),x)\overline{a_\mu(H(z),y)}\chi(z),$ where $\chi(z)$ is a smooth bump function supported in the ball $B^{d-1}(0,\delta)$,
then it suffices to estimate the integral:
\begin{equation}\label{IK'}
I=\la^\frac{d-1}2\mu^\frac{d-1}2\iiint e^{i\la d_g(H(z),x)-i\mu d_g(H(z),y)}b(z,x,y)f(x)\overline{g(y)}\,dzdydx.
\end{equation}
If we let $\Phi(z,x,y)=d_g(H(z),x)-\epsilon d_g(H(z),y),$ with $0<\epsilon=\mu/\la<c$ as before, then it suffices to estimate the kernel
   \begin{equation}
K_1(x,y)=\int e^{i\la \Phi(z,x,y)}b(z,x,y)\,dz.
\end{equation}
Meanwhile, by Lemma \ref{Gauss'}, if we write $x=(x_1,x')$ and $y=(y_1,y')$, then
\begin{equation}
\nabla_z\Phi(0,x,y)=\frac{x'}{|x|}-\epsilon\frac{y'}{|y|},
\end{equation}
and
\begin{equation}
\nabla_z\Phi(z,x,y)=\frac{x'-z}{|x|}-\epsilon\frac{y'-z}{|y|}+O(|z|),
\end{equation}
which vanishes if \begin{equation}\label{z_0}\frac{x'-z}{|x|}=\epsilon\frac{y'-z}{|y|}+O(|z|).\end{equation}
On the other hand, if we let $z_0(x,y)$ denote the point that $\nabla_z\Phi(z_0,x,y)=0$, then
\begin{equation}\label{z_02}
\nabla^2_z\Phi(z_0,x,y)=-\frac{1}{|x|}I_{(n-1)\times(n-1)}+\epsilon\frac{1}{|y|}I_{(n-1)\times(n-1)}+O(1).
\end{equation}
On account of \eqref{z_0}, it is clear that the determinant of $\nabla^2_z\Phi(z_0,x,y)$ is bounded below by an absolute constant, providing that $\eps_0$ is small enough and $c_1,\,c_2$ are sufficiently close to 1. Therefore, without loss of generality, we may assume that $z_0(x,y)$  is the only critical point of $\Phi$ in $B^{d-1}(0,2\delta)$. Indeed, if such a critical point does not exist, integration by parts yields much better bounds. Note also that $z_0(x,y)$ will be a smooth function of $x$ and $y$. Now the method of stationary phase  yields that
\begin{equation}\label{stationary'}
K_1(x,y)=\lambda^{-\frac{n-1}{2}}e^{i\la\Phi(z_0(x,y),x,y)}\tilde b(x,y),
\end{equation}
Now  let us write $x$ in polar coordinates, $x=r\omega$ with $\omega\in S^{n-1}$. If we freeze the $r$ variable, and use Cauchy-Schwarz inequality,  we have
\begin{align*}
|I|^2&=\mu^{n-1}\left|\iiint e^{i\la\Phi(z_0(r\omega,y),r\omega,y)}\tilde b(r\omega,y) f(r\omega)\overline{g(y)}\,dydrd\omega\right|^2\\
&\le\mu^{n-1}\int\left|\iint e^{i\la\Phi(z_0(r\omega,y),r\omega,y)}\tilde b(r\omega,y) f(r\omega)\overline{g(y)}\,dyd\omega\right|^2dr\\
&\le\mu^{n-1}\|g\|_{L^2(M)}^2\int\left|\int e^{i\la\Phi(z_0(r\omega,y),r\omega,y))} \tilde b(r\omega,y) f(r\omega)\,d\omega\right|^2dy.
\end{align*}
It then suffices to bound    
\begin{equation}
II=\iiint e^{i\la[\Phi(z_0(r\omega,y),r\omega,y)-\Phi(z_0(r\omega',y),r\omega',y)]}\overline{\tilde b(r\omega',y)}\tilde b(r\omega,y)f(r\omega)\overline{f(r\omega')}\,d\omega d\omega'dy.\end{equation}
uniformly in $r$.

For the sake of simplicity, let
\[B_r(\omega,\omega',y)=\tilde b(r\omega,y)\overline{\tilde b(r\omega',y)},\]
and 
\begin{align*}
\Psi_r(\omega,\omega',y)&=\epsilon^{-1}[\Phi(z_0(r\omega,y),r\omega,y)-\Phi(z_0(r\omega',y),r\omega',y)]\\&=\epsilon^{-1}d_g(H(z_0),r\omega)- d_g(H(z_0),y)-[\epsilon^{-1}d_g(H(z_0),r\omega')- d_g(H(z_0),y)].
\end{align*}
It suffices to consider the kernel:
   \begin{equation}\label{5'}
K_r(\omega,\omega')=\int e^{i\mu\Psi_r(\omega,\omega',y)}B_r(\omega,\omega',y)\,dy.
\end{equation}
For the sake of simplicity, fixing $\omega$ and $y$, we may assume $z_0(r\omega,y)=0$, at the cost of a smooth change of variable. In order to reduce to the 2-dimensional case, let us fix a direction $e_2$ on the hyperplane $\{y=(y_1,y'):y_1=0\}$, say $(0,1,0,\ldots,0)$, and  consider the integral of $K_r$ in $y$ over the 2-plane $Y$ spanned by $e_1$ and $e_2$:
   \begin{equation}
K_r^Y(\omega,\omega')=\int e^{i\mu\Psi_r(\omega,\omega',(y_1,y_2,0,\ldots,0))}B_r(\omega,\omega',(y_1,y_2,0,\ldots,0)))\,dy_1dy_2.
\end{equation}
Now if we let $\omega^*$, ${\omega'}^*$ the projections of $\omega$, $\omega'$ onto $Y$ respectively, and $\omega^\perp$ ${\omega'}^\perp$ the projections of $\omega$ on to $Y^\perp$, then by the same argument as in \S 4, we see that
   \begin{equation}
|K_r^Y(\omega,\omega')|\le C_N(1+\mu|\omega^*-{\omega'}^*|)^{-\frac12}(1+\mu|\omega^\perp-{\omega'}^\perp|)^{-N},
\end{equation}
where the decay coming from the perpendicular direction is due to the fact that $$\partial_y\Psi_r(\omega,\omega',(y_1,y_2,0,\ldots,0))=-\nabla_yd_g(H(z_0(r\omega,y)),y)+\nabla_yd_g(H(z_0(r\omega',y)),y).$$ has absolute value bounded below by $|\omega^*-{\omega'}^*||y_1|+|\omega^\perp-{\omega'}^\perp|, $\footnote{Indeed, in our coordinates, $\omega^\perp$=0, and therefore  $|\omega^*-{\omega'}^*|=|\omega-{\omega'}^*|$, and $|\omega^\perp-{\omega'}^\perp|=|{\omega'}^\perp|$.} where, unlike \eqref{2d}, the second term is independent of $|y_1|$.
Therefore, if $\int dY$ denotes the integral over the $n-2$ dimensional family of 2-planes, we have
   \begin{align*}\label{7'}
&|II|=\left|\iiint K^Y_r(\omega,\omega')f(r\omega)\overline{f(r\omega')}\,d\omega d\omega' dY \right|\\
&\le C\|f\|_{L^2(d\omega)}\Big\|\iint (1+\mu|\omega^*-{\omega'}^*|)^{-\frac12}(1+\mu|\omega^\perp-{\omega'}^\perp|)^{-N} f(r\omega')\,d\omega'dY \Big\|_{L^2(d\omega)}\\
&\le C\|f\|_{L^2(d\omega)}^2\Big\|\int (1+\mu|\omega^*-{\omega'}^*|)^{-\frac12}(1+\mu|\omega^\perp-{\omega'}^\perp|)^{-N} d\omega' \Big\|_{L^1(d\omega dY)}\\
&\le C\mu^{-(n-2)-\frac12}\|f\|_{L^2(d\omega)}^2,
\end{align*}
which implies that $|I|\le\mu^\frac12\|g\|_{L^2(M)}^2\|f\|_{L^2(M)}^2$, finishing the proof of Theorem \ref{thm4}.
\section{Further Comments} 
\subsection{Inner Product of Eigenfunction over Subdomain}
Another closely related problem, suggested by Professor Allan Greenleaf, is to prove an quasi orthogonality type result for the restriction of eigenfunctions to a subdomain $\Omega\subset M$ with smooth boundary $\partial \Omega$, i.e., estimating the integral
\begin{equation}
\label{problem9}
\int_{\Omega}e_\lambda \overline{e_\mu}\, dx.
\end{equation}
This problem originated from engineering applications, which all had a simple idea behind them:  in a fixed subdomain, as the frequencies $\lambda$, $\mu$  go to infinity, in principle, the  two restricted eigenfunctions $e_\lambda|_\Omega$ and $e_\mu|_\Omega$ should show some degree of orthogonality, due to their quick oscillations. In the special case that both $e_\lambda$ and $e_\mu$ vanishes on the boundary of $\partial\Omega$, they are  perfectly orthogonal over $\Omega$ if $\lambda\neq \mu$. Indeed, in this special case, 
$e_\lambda|_\Omega$ and $e_\mu|_\Omega$ are Dirichlet eigenfunctions over $\Omega$.

A simple observation is that on $\mathbb T^2$, if we take $\mu=\lambda+2k+1$ for any fixed integer $k$, the inner-product of $\sin \mu x$ and $\cos \lambda x$ over the subdomain $[0,\pi]\times[0,2\pi]$ will have a uniform lower bound depending on $k$.  This observation indicates that we need to assume some separation for $\lambda$ and $\mu$. The same example also indicates that the best decay we can get, without any further geometric assumption on the subdomain or  manifold, will be no better than $|\lambda-\mu|^{-1}$. Furthermore, on  the flat torus $\mathbb T^2=[0,2\pi]^2$, it is not hard to see that if the boundary $\partial \Omega$ has non-vanishing curvature, then we can get stronger decay.

If we integrate by parts twice, we can see that this inner product $\eqref{problem9}$ is closely connected to the the restriction of eigenfunctions to the boundary $\partial \Omega$. Indeed, if $0<\mu<c\la$ for some $c<1$, by Green's formula, we have
\begin{equation}(\la^2-\mu^2)\int_{\Omega}e_\lambda \overline{e_\mu}\, dx=\int_{\partial\Omega}e_\lambda \overline{\frac{d}{d\nu}e_\mu}\, d\sigma-\int_{\partial\Omega}\frac{d}{d\nu}e_\lambda \overline{e_\mu}\, d\sigma,\  
\end{equation}
which have apparent ties to inner product of eigenfunctions over curves. Indeed, by our proof of Theorem \ref{thm1} and \ref{thm4} one sees that 
\begin{equation}\left|\int_{\partial\Omega}e_\lambda \overline{\frac{d}{d\nu}e_\mu}\, d\sigma\right|\le C\mu^\frac54,
\end{equation}
as well as that
\begin{equation}\label{dn}\left|\int_{\partial\Omega}\frac{d}{d\nu}e_\lambda \overline{e_\mu}\, d\sigma\right|\le C\la\mu^\frac14,
\end{equation}
which implies that
\begin{equation}\label{domain}\left|\int_{\Omega}e_\lambda \overline{e_\mu}\, dx\right|\le C\frac{\la\mu^\frac14}{\la^2 -\mu^2}.
\end{equation}
 On the other hand, \eqref{dn} also trivially follows from the $L^2$-restriction bound \eqref{rest} of Burq, G{\'e}rard and Tzvetkov \cite{BGT}, and a result of Christianson, Hassell, and Toth \cite{Andrew} on the restriction of normal derivatives of eigenfunctions, which says that
\begin{equation}\Big\|\frac{d}{d\nu}e_\lambda\Big\|_{L^2{(\partial\Omega)}}\le C\la.
\end{equation}

Considering the examples on $\mathbb T^2$, the best bound we can hope for is 
\begin{equation}\label{domain'}\left|\int_{\Omega}e_\lambda \overline{e_\mu}\, dx\right|\le C\frac{1}{\la -\mu}.
\end{equation}
Therefore, it would be interesting to see if we can prove that under some suitable geometric assumptions,
\begin{equation}\left|\int_{\partial\Omega}\frac{d}{d\nu}e_\lambda \overline{e_\mu}\, d\sigma\right|\le C\la,\quad\text{and }\left|\int_{\partial\Omega}e_\lambda \overline{\frac{d}{d\nu}e_\mu}\, d\sigma\right|\le C\la,
\end{equation}
which would imply \eqref{domain'} for the torus.
\subsection{Improving \eqref{inner}, \eqref{inner'} and \eqref{period} for  Manifolds with Non-positive Curvature}
It would be interesting to see if we can improve \eqref{inner} or \eqref{period} if our surface is non-positively curved. First of all, as mentioned in \S 4, \eqref{inner} is not optimal on $\mathbb T^2$, which raises the question whether we can improve \eqref{inner} for non-positively curved surfaces. Such an improvement would potentially unify the improved $L^2$-restriction bound of Blair and Sogge \cite{BSTop}, and the improved period integral bound of Sogge Xi and Zhang \cite{gauss}. To be more specific, for $0<\mu<c\la$, we seek a bound in the  form that
	\begin{equation}\label{inner''}
\Big|\int_{\gamma}e_\lambda \overline{e_\mu}\, ds\Big|\le C \frac{\mu^\frac14}{(\log\la)^\frac12},
\end{equation}
for s closed geodesic $\gamma$ on a surface with suitable curvature assumptions. On the other hand, for the generalized periods, it is conjectured in \cite{Rez} that if $M$ is a compact hyperbolic surface, $\gamma$ is a close geodesic, then for any $\eps$, $\nu<c\la$ for some $0<c<1$,
	\begin{equation}\label{period'}
\Big|\int_{\gamma}e_\lambda e^{-i\nu s}\, ds\Big|\le C_\eps \la^{-\frac12+\eps}.
\end{equation}
Even though  \eqref{period'} seems unreachable by current methods, if we exploit the argument in \cite{gauss} it seems quite plausible to get a log-type improvement over the generalized period bounds of Reznikov \eqref{GP} for compact hyperbolic surfaces, in the sense that
	\begin{equation}\label{Gperiod'}
\Big|\int_{\gamma}e_\lambda e^{-i\nu s}\, ds\Big|\le C {(\log\la)}^{-\frac12},
\end{equation}
for $\nu<c\la$.\ \footnote{Shortly after posting this paper, the author proved \eqref{Gperiod'} by a refinement of the arguments in \cite{gauss}, see \cite{GP}.}
\bibliography{EF}{}
\bibliographystyle{alpha}
\end{document}